\documentclass[english,11pt]{amsart}

\usepackage{babel}
\usepackage{amstext}
\usepackage{amsmath}
\usepackage{amsfonts}
\usepackage{latexsym}
\usepackage{ifthen}
\usepackage{amsthm}
\usepackage{amssymb}
\usepackage[all]{xy}
\usepackage{textcomp}

\newcommand{\N}{\mathcal{N}_1}
\newcommand{\Supp}{\operatorname{Supp}}
\newcommand{\sing}{\operatorname{sing}}
\newcommand{\NE}{\operatorname{NE}}
\newcommand{\Exc}{\operatorname{Exc}}
\newcommand{\Locus}{\operatorname{Locus}}
\newcommand{\codim}{\operatorname{codim}}
\newcommand{\reg}{\operatorname{reg}}
\newcommand{\Hilb}{\operatorname{Hilb}}
\newcommand{\red}{\operatorname{red}}
\newcommand{\Spec}{\operatorname{Spec}}

\newtheorem{thm}{Theorem}[section]
\newtheorem{lem}[thm]{Lemma}
\newtheorem{cor}[thm]{Corollary}
\newtheorem{pro}[thm]{Proposition}

\theoremstyle{definition}
\newtheorem{nr}[thm]{}
\newtheorem{rmk}[thm]{Remark}
\newtheorem{ex}[thm]{Example}
\newtheorem{defi}[thm]{Definition}

\numberwithin{equation}{section}

\begin{document}
\title{On the Picard number of singular Fano varieties}
\author{Gloria Della Noce}
\address{Dipartimento di Matematica \\ Universit\`a  di Pavia \\ via Ferrata, 1 27100 Pavia - Italy}
\email{gloria.dellanoce@unipv.it}

\date{February 6, 2012}

\begin{abstract}
Let $X$ be a $\mathbb{Q}$-factorial Gorenstein Fano variety. Suppose that the singularities of $X$ are canonical and that the locus where they are non-terminal has dimension zero. Let $D \subset X$ be a prime divisor. We show that $\rho_X - \rho_D \leq 8$. Moreover, if $\rho_X - \rho_D \geq 4$, there exists a finite morphism $\pi:X \to S \times Y$, where $S$ is a surface with $\rho_S \leq 9$.

\noindent As an application we prove that, if $\dim(X)=3$, then $\rho_X \leq 10$.

\end{abstract}

\maketitle

\bibliographystyle{amsalpha}

\section*{Introduction}

Let $X$ be a (possibly singular) Fano variety, \textit{i.e.} a normal variety whose anticanonical divisor has a multiple which is Cartier and ample. Let $D \subset X$ be a prime divisor. We denote by $\N(X)$ (resp. $\N(D)$) the vector space of real one-cycles in $X$ (resp. in $D$), modulo numerical equivalence. By definition, $\dim \N(X)=\rho_X$ is the Picard number of $X$, and similarly for $D$. The inclusion $i:D \hookrightarrow X$ induces a linear map $i_*:\N(D)\to \N(X)$; let us define
\[
 \N(D,X)=i_*\N(D) \subseteq \N(X).
\]
Thus $\N(D,X)$ is the subvector space of $\N(X)$ whose elements are the numerical equivalence classes of one-cycles contained in $D$. Notice that the dimension of this space could be strictly smaller that the Picard number of $D$, because $i_*$ does not need to be injective.

In this paper we are interested in finding an upper bound, non depending on $D$, for the codimension of $\N(D,X)$ in $\N(X)$. We then show how, under additional assumptions, the knowledge of this bound gives us information on the geometry of $X$ and its Picard number.

This problem was first introduced by C. Casagrande in \cite{cas10}, where the author studied the smooth case. Her main result is the following:

\begin{thm}\label{cinzia}\cite[Theorem 1.1]{cas10}
Let $X$ be a Fano manifold. 
For every prime divisor $D\subset X$,  we have 
\[
\rho_X-\rho_D\leq \codim\N(D,X)\leq 8.
\]
Moreover, suppose that there exists a prime divisor $D$ with
$\codim\N(D,X)\geq 4$. Then $X\cong
S\times Y$, where $S$ is a Del Pezzo surface
with $\rho_S\geq\codim\N(D,X)+1$, 
 and $D$ dominates $Y$ under the projection.
\end{thm}

In this paper we study what happens if $X$ is allowed to have mild singularities. Our approach is the same of Casagrande's paper and our main result is:

\begin{thm}\label{teo_princ}
Let $X$ be a $\mathbb{Q}$-factorial Gorenstein Fano variety of dimension $n$, with canonical singularities, and with at most finitely many non-terminal points. Then for every prime divisor $D \subset X$
\[
 \rho_X-\rho_D\leq \codim\N(D,X)\leq 8.
\]
Furthermore, if there exists a prime divisor $D \subset X$ such that $\codim \N(D,X)\geq 4$, there is a finite morphism
\[
\pi: X \rightarrow S \times Y,
\]
where $Y$ is a normal variety of dimension $n-2$ with rational singularities, and $S$ is a normal surface with rational quotient singularities, such that $9 \geq \rho_S \geq \codim \N(D,X) + 1$. Moreover, $\rho_X=\rho_S + \rho_Y$.

\end{thm}

The most important consequence of Theorem \ref{teo_princ} concerns the case of dimension $3$, where we find an explicit bound for the Picard number of $X$:

\begin{thm}\label{dim3} Let $X$ be a three-dimensional $\mathbb{Q}$-factorial Gorenstein Fano variety whose singularities are canonical and isolated. Then $\rho_X \leq 10$.
\end{thm}

In the setting of Theorem \ref{teo_princ}, the Picard number of $X$ is a topological invariant, since it coincides with the second Betti number of $X$. In fact Kodaira vanishing (see \cite[Theorem 2.70]{km} for the singular version) implies that $H^i(X,\mathcal{O}_X)=0$ for every $i>0$. Considering now the long exact sequence in cohomology induced by the exponential sequence, we see that there is an isomorphism between $H^2(X,\mathbb{Z})$ and the Picard group of $X$, whose rank is $\rho_X$.

In the smooth case, it is well known (\cite{kmm2}) that in every dimension there are only finitely many families of Fano varieties; in particular the Picard number is bounded in any dimension. In dimension $2$ this bound equals $9$, in dimension $3$ is $10$, and in higher dimensions only some partial results are known. 

In the singular case the maximal values for the Picard number are known in some particular low-dimensional cases. It is well known that, if $X$ is a Del Pezzo surface with canonical singularities, its Picard number cannot exceed $9$.

If $X$ is a Gorenstein Fano variety of dimension $3$ with terminal singularities, it can be deformed to a smooth Fano $3$-fold (\cite[Theorem 11]{nam}) and the Picard number is preserved under this deformation (\cite[Theorem 1.4]{jr}). Thus the Picard number of $X$ does not exceed $10$.

If, instead, $X$ is a Fano Gorenstein $3$-fold with canonical isolated singularities, then $X$ is not, in general, a deformation of a smooth Fano $3$-fold. An example is given by the weighted projective space $\mathbb{P}(1,1,1,3)$ (see \cite[Example 1.4]{prok}).

In general, it is clear that, in order to get $9$ (resp. $10$) as a suitable bound for the Picard number of a Del Pezzo surface (resp. Fano threefold), some restrictions on the singularities are necessary. In Example \ref{esempio}, we exhibit a Del Pezzo surface $S$ with non-canonical singularities and index $15$ (recall that the index is the smallest integer $r$ such that $rK_S$ is a Cartier divisor), whose Picard number is $10$. Similarly, $S \times \mathbb{P}^1$ is a non-Gorenstein Fano threefold with non-canonical singularities and Picard number $11$. The surface $S$ was found using the classification of toric log Del Pezzo surfaces of index at most $16$ in \cite{kkn}; the list of such surfaces is available in the Graded Ring Database \cite{bro}.
\medskip

The paper is organized as follows. The first section is a self-contained part devoted to the study of $K$-negative birational contractions with at most one-dimensional fibers defined on varieties with mild singularities. We present a result that will be used in the proofs of Theorems \ref{teo_princ} and \ref{dim3}. Its proof is based on the theorem of existence of flips and its main point is the study of the bahaviour of the discrepancies under the flip.

The second section is entirely devoted to the proof of some preliminary results for Theorem \ref{teo_princ}. In subsection 2.1, we collect some results concerning Mori programs for Fano varieties. In fact, from \cite[Corollary 1.3.2]{bchm}, we know that $\mathbb{Q}$-factorial Fano varieties with canonical singularities are Mori dream spaces (see \cite{hk}); in particular we can run a Mori program for every divisor. Let $X$ be such a Fano variety and $D \subset X$ a prime divisor. We show the existence of a ``special'' Mori program for the divisor $-D$ and we study its properties. In subsection 2.2, we define an invariant of $X$ which was first introduced by C. Casagrande in \cite{cas10}. Under an assumption on such an invariant, we study what happens when we run a special Mori program for $-D$, when $D \subset X$ is a prime divisor such that $\dim\N(D,X)$ is minimal.

The third section is the main body of the paper and contains the proof of Theorem \ref{teo_princ}. After noting that it is sufficient to prove the theorem under the assumption of existence of a divisor $D \subset X$ with $\codim \N(D,X) \geq 4$, we use the results of the second section in order to construct the finite morphism $\pi$ of the statement.

In the fourth section we prove Theorem \ref{dim3} and we make some further remarks.

\medskip

\noindent\textbf{Notation and terminology}

\medskip
\noindent We work over the field of complex numbers.\\
Let $X$ be a normal projective variety.\\
$X$ is called a \textit{Fano variety} if $-K_X$ admits a multiple which is Cartier and ample.\\
We denote by $X_{\reg}$ the smooth locus of $X$ and by $X_{\sing}$ its singular locus.\\
Unless otherwise stated, any divisor will be a \textit{Weil} divisor.\\
A divisor is called \textit{$\mathbb{Q}$-Cartier} if it admits a multiple which is Cartier.\\
$X$ is called \textit{$\mathbb{Q}$-factorial} is every divisor is $\mathbb{Q}$-Cartier.\\
The \textit{index} of $X$ is the smallest integer $r$ such that $rK_X$ is a Cartier divisor.\\
For the definitions and properties of terminal/canonical/log-terminal/... singularities, we refer the reader to \cite{km}.\\
If $X$ has canonical singularities, it is said to be \textit{Gorenstein} if its index is one. A point $p \in X$ is a Gorenstein point if $X$ is Gorenstein in a neighborhood of $p$. The subset of $X$ of its Gorenstein points is open and is called \textit{Gorenstein locus}.\\

\noindent \emph{We denote by $NT(X)$ the closed subset of $X$ made up by canonical non-terminal singularities.}\\ \label{glo}

\noindent $\N(X)$ is the vector space of one-cycles with real coefficients, modulo numerical equivalence.\\
$\mathcal{N}^1(X)$ is the vector space of $\mathbb{Q}$-Cartier divisors with real coefficients, modulo numerical equivalence.\\
$\dim \N(X) = \dim \mathcal{N}^1(X):=\rho_X$ is the \textit{Picard number} of $X$.\\
Let $D \subset X$ be a $\mathbb{Q}$-Cartier divisor. We denote by $[D]$ its numerical equivalence class in $\mathcal{N}^1(X)$.\\
We define $D^{\perp}:=\{\gamma \in \N(X) \ | \ \gamma \cdot D = 0 \}$.\\
Let $C \subset X$ be a one-cycle. We denote by $[C]$ the numerical equivalence class of $C$ in $\N(X)$, by $\mathbb{R}[C]$ the one-dimensional vector space it spans in $\N(X)$ and by $\mathbb{R}_{\geq 0}[C]$ the corresponding ray.\\
The intersection product between $D$ and $C$ is denoted by $D \cdot C$.\\
Let $C \subset X$ be a one-dimensional subscheme. For semplicity we still denote by $D \cdot C$ the intersection product of $D$ with the one-cycle associated to $C$.\\
$\NE(X) \subset \N(X)$ is the convex cone generated by classes of effective curves and $\overline{\NE}(X)$ is its closure.\\
An \textit{extremal ray} $R$ of $X$ is a one-dimensional face of $\overline{\NE}(X)$. We denote by $\Locus(R) \subseteq X$ the union of curves whose class belong to $R$.\\

\noindent A \textit{contraction} of $X$ is a projective surjective morphism with connected fibers $\varphi:X \to Y$ onto a projective normal variety $Y$.\\
The push-forward of one-cycles defined by $\varphi$ induces a surjective linear map $\varphi_*:\N(X) \to \N(Y)$.\\
Define $\NE(\varphi):=\NE(X) \cap \ker(\varphi_*)$.\\
We denote by $\Exc(\varphi)$ the \textit{exceptional locus} of $\varphi$, \textit{i.e.} the locus where $\varphi$ is not an isomorphism.\\
We say that $\varphi$ is \textit{of fiber type} if $\dim(X) > \dim(Y)$, otherwise $\varphi$ is birational.\\
$\varphi$ is called \textit{elementary} if $\dim(\ker(\varphi_*))=1$. In this case we say that $\varphi$ is \textit{divisorial} if $\Exc(\varphi)$ is a prime divisor of $X$ and \textit{small} if its codimension is greater than $1$.\\
A contraction of $X$ is called \textit{$K_X$-negative} (or simply \textit{$K$-negative}) if the canonical divisor $K_X$ of $X$ is $\mathbb{Q}$-Cartier and $-K_X \cdot C > 0$ for every curve $C$ contracted by $\varphi$.\\

\noindent If $Z \subset X$ is a closed set and $i:Z \to X$ is the inclusion map, we set
\[
\N(Z,X):=i_*\N(Z) \subseteq \N(X) \ \ \textmd{and} \ \ \overline{\NE}(Z,X):=i_*(\overline{\NE}(Z)) \subseteq \overline{\NE}(X). 
\]
(Notice that $\overline{\NE}(Z,X) \subseteq \N(Z,X) \cap \overline{\NE}(X)$, but equality does not hold in general.)
\medskip

\noindent We denote by $\Hilb(X)$ the Hilbert scheme of $X$ and by $[Z] \in \Hilb(X)$ the point which corresponds to the subscheme $Z \subset X$.\\
If $Z \subset X$ is a subscheme of $X$, we denote by $Z_{\red}$ the corresponding reduced scheme.
\medskip

\noindent Given a (holomorphic or algebraic) vector bundle $\pi:E \to X$, we denote by $p:\mathbb{P}(E) \to X$ the associated projective bundle.

\medskip

\noindent \textbf{Acknowledgements.} I wish to thank my PhD advisor Cinzia Casagrande for having introduced me to the subject and for her constant guidance and support. This work would not have been possible without her several suggestions.

\section{A result on $K$-negative contractions}

In this section we present a result about $K$-negative birational contractions whose fibers are at most one-dimensional. When the ambient variety is smooth, it is well-known that the exceptional locus of such contractions has codimension one. This is a consequence of a more general result proved by J. Wi\'{s}niewski (\cite[Theorem 1.1]{W}). Moreover the following result holds:

\begin{lem}\label{lem_ando}\cite[Theorem 2.3 and its proof]{ando}
Let $X$ be a smooth variety and let $\varphi:X \rightarrow Y$ be a $K_X$-negative elementary divisorial contraction with fibers of dimension $\leq 1$. Call $E$ the exceptional divisor of $\varphi$. Then $Y$ and $\varphi(E)$ are smooth and $\varphi$ is the blow-up of $Y$ along $\varphi(E)$. Moreover $\varphi(E)$ is isomorphic to the connected component of $\Hilb(X)$ which contains the point corresponding to a non-trivial fiber of $\varphi$, and $\varphi_{|E}$ is the restriction to such a component of the universal family over $\Hilb(X)$.
\end{lem}

We show here that a similar result holds if $X$ is allowed to have mild singularities. This is probably well-known to experts in the field; our argument to prove that the exceptional locus has pure codimension one when $X$ has terminal singularities, is based on the theorem of existence of flips and is the same as in \cite[Example 1]{shok}. The generalization to the case of non-terminal isolated singularities is obtained thanks to the theorem of existence of a crepant terminalization. For clarity, we give here a complete proof adapted to our context.

\begin{thm}\label{teo}
 Let $X$ be a normal projective variety of dimension $n\geq 3$ with canonical singularities and with at most finitely many non-terminal points (\textit{i.e.} $\dim(NT(X))\leq0$, where $NT(X)$ is defined in page \pageref{glo}). Let $\varphi:X \rightarrow Y$ be a birational $K_X$-negative contraction whose fibers are at most one-dimensional. Suppose, moreover, that the exceptional locus of $\varphi$ is contained in the Gorenstein locus $G$ of $X$. Then:
\begin{enumerate}
 \item every non-trivial fiber of $\varphi$ is irreducible, has no multiple one-dimensional components and its reduced structure is isomorphic to $\mathbb{P}^1$. Moreover the general non-trivial fiber is smooth, \textit{i.e.} is isomorphic to $\mathbb{P}^1$ as scheme;
 \item let $R_1,\ldots,R_s$ be the extremal rays of $\NE(\varphi) \subseteq \N(X)$. Then, for every $i=1,\ldots,s$, the contraction of $R_i$ is divisorial. Moreover, if $E_i:=\Locus(R_i)$, then $E_1,\ldots,E_s$ are pairwise disjoint prime divisors and
\[
 \Exc(\varphi)=\bigcup_{i=1}^s E_i;
\]
 \item $Y$ has canonical singularities and $\dim(NT(Y)) \leq 0$;
 \item there exists a closed subset $T \subset Y$ with $\codim T\geq 3$ such that $Y \smallsetminus T \subseteq Y_{\reg}$, $\codim \varphi^{-1}(T) \geq 2$, $X \smallsetminus \varphi^{-1}(T) \subseteq X_{\reg}$ and
\[
 \varphi_{|{X \smallsetminus \varphi^{-1}(T)}}: X \smallsetminus \varphi^{-1}(T) \rightarrow Y \smallsetminus T
\]
is the simultaneous blow-up of the $(n-2)$-dimensional pairwise disjoint smooth varieties $\varphi(E_i) \cap (Y \smallsetminus T)$. In particular
\[
 K_X = \varphi^*(K_Y) + E_1 + \cdots + E_s;
\]
 \item for every $i=1,\ldots,s$, let $f_i \subset X$ be an irreducible curve such that $[f_i]\in R_i$. Then
\[
 K_X \cdot f_i = E_i \cdot f_i = -1.
\]

\end{enumerate}
\end{thm}

\begin{proof}
\begin{nr}\textbf{Proof of (1).} The first assertion of (1) follows from \cite[Theorem 1.10 (i)]{AW}. The general fiber is smooth because, by our assumptions, $\dim(X_{\sing}) \leq n-3$ (see \cite[Corollary 5.18]{km}) and hence $X_{\sing}$ cannot dominate $\varphi(\Exc(\varphi))$.
\end{nr}
\medskip
The proof of (2) requires some preliminary steps.

\begin{nr}\label{teo0}\textbf{$\Locus(R_1),\ldots,\Locus(R_s)$ are pairwise disjoint. Moreover, if they are prime divisors, $\Exc(\varphi)$ is their union.}
Fix an index $i \in \{1, \ldots, s\}$ and let $\varphi_i:X \rightarrow Y_i$ be the contraction of $R_i$; then every non-trivial fiber $f_i$ of $\varphi_i$ is contracted by $\varphi$, hence is contained in a fiber of $\varphi$, say $f$. Applying (1) to both $\varphi$ and $\varphi_i$, we see that the reduced structures of $f$ and $f_i$ are irreducible, hence they coincide. In particular every fiber of $\varphi_i$ is disjoint from any fiber of $\varphi_j$ whenever $i \neq j$.

Suppose now that every $\varphi_i$ is divisorial with exceptional divisors $E_i$ and let $C$ be an irreducible curve contracted by $\varphi$. Then $[C] \in \NE(\varphi)=R_1 + \cdots + R_s$ and we can write
\[
 [C] = \sum_{i=1}^s \lambda_i [f_i],
\]
where for every $i=1, \ldots s$, $[f_i]$ generates $R_i$, $\lambda_i \geq 0$ and $\lambda_j>0$ for at least one $j$. Intersecting with $E_j$, we get:
\[
C \cdot E_j = \lambda_j (f_j \cdot E_j) <0,
\]
so that $C \subset E_j$. Hence $\Exc(\varphi)= \bigcup_{i=1}^s \Exc(\varphi_i)$. 
\end{nr}
\begin{nr}\label{teo1}\textbf{If $X$ has terminal singularities and $\varphi$ is elementary, $\Exc(\varphi)$ cannot be one-dimensional.} By contradiction suppose that this is the case. Let $\varphi^+: X^+ \rightarrow Y$ be the flip of $\varphi$, which exists by \cite[Corollary 1.4.1]{bchm}; denote by $A^+ \subset X^+$ its exceptional locus and let $\Phi: X \dashrightarrow X^+$ be the resulting birational map. Since $\dim(\Exc(\varphi))=1$, by \cite[Lemma 5.1.17]{kmm} (notice that in \cite{kmm} the $\mathbb{Q}$-factoriality is required, but this assumption is actually not necessary), the dimension of $A^+$ is $n-2$.

Let us consider a common smooth resolution $Z$ of $X$ and $X^+$:
\[
\xymatrix{
& \ar[dl]_f Z \ar[dr]^g & \\
X \ar[dr]_{\varphi} \ar@{-->}[rr]^\Phi & & X^+ \ar[dl]^{\varphi^+} \\
& Y. &
}
\]
According to Hironaka's results, we can suppose the exceptional locus of $f$ and $g$ to be of pure codimension $1$. We can write:
\begin{equation}
\label{can}
 K_Z \sim f^*K_X + \sum_{i=0}^{k}a_i E_i \sim g^* K_{X^+} + \sum_{i=0}^k b_i E_i,
\end{equation}
where $E_0,\ldots,E_k \subset Z$ are the exceptional divisors and $0 \leq a_i \leq b_i$ for every
$i= 0, \ldots, k$ (for the second inequality see \cite[Lemma 3.38]{km}); in particular $X^+$ has terminal singularities. Let us notice, moreover, that $\sum_{i=0}^{k}a_i E_i$ is an integral Cartier divisor on $f^{-1}(G)$ (ehere $G$ is the Gorenstein locus of $X$); in particular the coefficient $a_i$ is an integral number whenever $E_i$ intersects $f^{-1}(G)$.

Let $\Lambda \subseteq A^+$ be an irreducible component of dimension $n-2$; then $X^+$ is smooth at the generic point of $\Lambda$.

Set $X_0^+:= X^+ \smallsetminus((X^+)_{\sing} \cup \Lambda_{\sing})$. Then $\Lambda_0 := \Lambda \cap X_{0}^{+}$ is non empty, smooth and has codimension $2$ in $X_{0}^{+}$, which is also smooth. Let us consider the blow-up $\pi: B \rightarrow X_{0}^{+}$ of $X_{0}^{+}$ along $\Lambda_0$; we have
\[
 K_B \sim \pi^*K_{X_{0}^{+}} + H_0,
\]
where $H_0 \subset B$ is the $\pi$-exceptional divisor. Set $Z_0:=g^{-1}(X_{0}^{+})$ and $g_0:=g_{|Z_0}$. By our assumptions $(g_0)^{-1}(\Lambda_0)$ is of pure codimension $1$ in $Z_0$, hence it is a Cartier divisor. Then, by the universal property of blow-up, the morphism $g_0$ factors through $\pi$:
\[
\xymatrix{
Z_0 \ar[dr]^{g_0} \ar[r]^h & B \ar[d]^{\pi}\\
& X_{0}^{+}.\\
}
\]
We can write:
\begin{equation}
\label{can0}
\begin{array}{ccl}
K_{Z_0} & \sim & h^* K_{B} + \sum_{i = 1}^r e_i F_i\\
& \sim & h^*(\pi^* K_{X_{0}^+} + H_0) + \sum_{i =1}^r e_i F_i\\
& \sim & (g_{0})^*K_{X_{0}^+} + F_0 + \sum_{i =1}^r(e_i + f_i)F_i,
\end{array}
\end{equation}
where $F_0 \subset Z_0$ is the transform of $H_0$ and, for every $i=1, \ldots, r$, $F_i$ is an $h$-exceptional prime divisor and $e_i$ and $f_i$ are integral numbers. In particular the closures $E_0, \ldots, E_r$ of $F_0, \ldots, F_r$ in $Z$ are $g$-exceptional divisors. Comparing \eqref{can} with \eqref{can0}, we finally deduce that $b_0=1$. By construction we know that $g(E_0) = \Lambda$. By \cite[Lemma 3.38]{km}, for every exceptional divisor $E_j$ such that $g(E_j) \subseteq A^+$, we have $a_j < b_j$. Thus $a_0 < b_0=1$; moreover $a_0$ is an integral number, because $E_0 \subseteq f^{-1}(G)$. Hence $a_0 = 0$, but this not possible because we are assuming that $X$ has terminal singularities.
\end{nr}
\begin{nr}\label{teo2}\textbf{If $X$ has terminal singularities and every irreducible component of $\Exc(\varphi)$ has dimension $1$ or $n-1$, then (2) holds.} For $i=1,\ldots,s$, let $\varphi_i$ be the contraction of the extremal ray $R_i$. By \ref{teo1}, we know that $\dim(\Exc(\varphi_i))\geq 2$ for every $i=1, \ldots,s$. Moreover, by \ref{teo0}, $\Exc(\varphi_i)$ is a union of irreducible components of $\Exc(\varphi)$. Then the $\varphi_i$'s are all divisorial and, by \ref{teo0}, their exceptional divisors $E_1,\ldots,E_s$ cover $\Exc(\varphi$). Hence we get a contradiction with the existence of a one-dimensional component. 
\end{nr}
\begin{nr}\label{teo3}\textbf{Proof of (2) when $X$ has terminal singularities.} Let us proceed by induction on $n$. 
When $n=3$, (2) holds by \ref{teo2}.

Let us suppose $n > 3$. Let us pick $H$ a general very ample divisor of $Y$. Set $\tilde{H}=\varphi^*(H)\subset X$ and let $\tilde{\varphi}$ be
the restriction of $\varphi$ to $\tilde{H}$. The morphism $\tilde{\varphi}$ is still a contraction, \textit{i.e.} it has connected fibers, and $H$ is normal.

By \ref{teo2}, the exceptional locus of $\varphi$ cannot be one-dimensional and this assures that $\tilde{\varphi}$ is not an isomorphism.  
The linear sistem of $\tilde{H}$ is base point free; then, by \cite[Lemma 5.17]{km}, we know that $\tilde{H}$ 
has terminal singularities. 
By the adjunction formula, given a curve $C$ contracted by $\tilde{\varphi}$, we have:
\[
 K_{\tilde{H}} \cdotp C =K_X \cdotp C + \tilde{H} \cdotp C = K_X \cdotp C + \varphi^{*}(H) \cdotp C < 0.
\]
Moreover, the points in $\tilde{H} \cap G$ are Gorenstein for $\tilde{H}$ and $\Exc \tilde{\varphi}$ is contained in the Gorenstein locus of $\tilde{H}$.

Thus $\tilde{H}$ and $\tilde{\varphi}$ satisfy all the assumptions of the theorem. By induction, every irreducible component of
$\Exc(\tilde{\varphi})$ has codimension one in $\tilde{H}$; then (2) holds by \ref{teo2}.
\end{nr}
\begin{nr}\textbf{Proof of (2): general case.} By \ref{teo0}, it is enough to prove (2) when $\varphi$ is elementary. By \cite[Corollary 1.4.3]{bchm}, there exists a birational morphism $\tau:\tilde{X}\to X$ such that:
\begin{enumerate}
 \item $\tilde{X}$ has $\mathbb{Q}$-factorial terminal singularities;
 \item given a resolution $f:Z\to X$ of singularities $X$, the $\tau$-exceptional divisors correspond to $f$-exceptional divisors with discrepancy zero;
 \item $K_{\tilde{X}}=\tau^{*}(K_X)$.
\end{enumerate}
Given an irreducible curve $C \subset \tilde{X}$ such that $\tau(C)$ is a curve contracted by $\varphi$, we have $K_{\tilde{X}}\cdot C<0$. Thus we can find an extremal ray $\tilde{R}\in \overline{NE}(\varphi \circ \tau)$ of $\overline{NE}(\tilde{X})$ such that $K_{\tilde{X}} \cdot \tilde{R}<0$; let $\tilde{\varphi}: \tilde{X}\to \tilde{Y}$ be its contraction. Then $\varphi \circ \tau$ factors through $\tilde{\varphi}$; in particular $\tilde{\varphi}$ is birational. Moreover its fibers are at most one-dimensional: if, by contradiction, $F$ is a fiber of $\tilde{\varphi}$ with $\dim(F)\geq 2$, then $\tau$ cannot be finite on $F$. Hence, by (3), there exists a curve $C \subset F$ with $K_{\tilde{X}}\cdot C=0$ and this is impossible because $[C]\in \tilde{R}$. A similar argument shows that $\Exc(\tilde{\varphi}) \subseteq \tau^{-1}(\Exc(\varphi))$; then $\Exc(\tilde{\varphi})$ is contained in the Gorenstein locus of $\tilde{X}$. We can thus apply \ref{teo3} to ${\tilde{\varphi}}$ and conclude that it is divisorial.

Let $D$ be the exceptional divisor of $\tilde{\varphi}$. Let us suppose that $D$ is exceptional for $\tau$. Then, by (1) and (2), we find an exceptional divisor $E \subset Z$ with discrepancy zero such that $f(E)=\tau(D)$. In particular $X$ has non-terminal singularities along $\tau(D)$, so that, by our assumptions, $\tau(D)$ is a point. In particular, by (3), we see that for every irreducible curve $C\subset D$ we have $K_{\tilde{X}}\cdot C=0$. This is not possible because $[C]\in \tilde{R}$; thus $\dim(\tau(D))=n-1$ and $\varphi$ is divisorial.
\end{nr}

\begin{nr}\textbf{Proof of (3), (4) and (5).} Statement (3) easily follows from \cite[Lemma 3.38]{km}: $Y$ has canonical singularities and might have non-terminal ones only at the images of the non-terminal singularities of $X$, which are finitely many by assumption. In particular $\dim(Y_{\sing}) \leq n-3$ by \cite[Corollary 5.18]{km}.

Let us prove (4). Since the general non-trivial fiber of $\varphi$ is contained in $X_{\reg}$, we have
\[
 \dim(\varphi([\Exc(\varphi)]^{\sing})) < \dim(\varphi(\Exc(\varphi)))=n-2,
\]
where $[\Exc(\varphi)]^{\sing} \subset \Exc(\varphi)$ is the closed subset made up by the non-trivial fibers which intersect the singular locus of $X$. Let us define $T:=Y_{\sing} \cup \varphi([\Exc(\varphi)]^{\sing})$, so that (4) follows from the smooth case applying Lemma \ref{lem_ando} $s$ times locally around each $E_i$.

Finally, (5) follows from the previous statements.
\end{nr} \end{proof}

\begin{defi}
 A contraction $\varphi$ as in Theorem \ref{teo} will be called \textit{of type $(n-1,n-2)^{eq}$} (the superscript standing for \textit{equidimensional}, referred to non-trivial fibers). If $\varphi$ is elementary with extremal ray $R$, the ray $R$ itself will be called \textit{of type $(n-1,n-2)^{eq}$}.
\end{defi}

The following example shows that the assumption on the non-terminal locus cannot be weakened and that an analogue of Theorem \ref{teo} is not true if we allow arbitrary canonical singularities.

\begin{ex}\label{Ex} For every $n \geq 3$, we construct an $n$-dimensional Fano variety $X$ with canonical singularities and with an elementary small contraction whose exceptional locus is one-dimensional.

Fix an integer $n \geq 3$ and consider  over $P:=\mathbb{P}^1 \times \mathbb{P}^{n-2}$ the projective bundle $Y = \mathbb{P}(\mathcal{E})$, where $\mathcal{E}$ is the rank-$2$ vector bundle 
\[
\mathcal{E}=\mathcal{O}_P \oplus \mathcal{O}_P(1,n-1).
\]
Let $p: Y \rightarrow P$ be the projection map and denote by $\mathcal{O}_Y(1)$ the tautological bundle. This is a nef but not ample line bundle; from the formula for the canonical bundle of $Y$
\begin{equation}\label{form}
\begin{array}{ccl}
\mathcal{O}_Y(K_Y) &  \sim & p^*(\mathcal{O}_P(K_P) \otimes \det(\mathcal{E})) \otimes \mathcal{O}_Y(-2) \\
& \sim & p^*(\mathcal{O}_P(-1,0)) \otimes \mathcal{O}_Y(-2),
\end{array}
\end{equation}
we see that $-K_Y$ is nef.

Let $E \simeq \mathbb{P}^1 \times \mathbb{P}^{n-2} \subset Y$ be the section of $p$ defined by the surjection of sheaves $\mathcal{E} \rightarrow \mathcal{O}_P$. The divisor $E$ has normal bundle $\mathcal{N}_{E\slash Y} = \mathcal{O}_P(-1,-(n-1))$. Let $\mathbb{P}^1 \simeq l_1 \subseteq \{\textmd{point}\} \times \mathbb{P}^{n-2}$ and $\mathbb{P}^1 \simeq l_2 = \mathbb{P}^1\times \{\textmd{point}\}$ be lines in $E$ contracted, respectively, by the first and the second projection of $E=\mathbb{P}^1 \times \mathbb{P}^{n-2}$. Then
\[
E \cdot l_1 = -(n-1) \ \ \textmd{and} \ \ E \cdot l_2 = -1.
\]

The variety $Y$ has three elementary contractions, corresponding to three generators of its nef cone. The first one is the morphism $p$, while the other two 
\[
f:Y \rightarrow X \ \ \textmd{and} \ \ g:Y \rightarrow Z
\]
are the contractions of the extremal rays generated, respectively, by $[l_1]$ and $[l_2]$. The contractions $f$ and $g$ are divisorial with exceptional divisor $E$, which is contracted to $\mathbb{P}^1$ by $f$ and to $\mathbb{P}^{n-2}$ by $g$.

We may write
\[
K_Y \sim f^*(K_X) + aE \sim g^*(K_Z) + bE,
\]
where $a$ and $b$ are some rational coefficients. Intersecting with $l_1$ and $l_2$ and using \eqref{form}, we get:
\[
0 = l_1 \cdot K_Y = l_1 \cdot (f^*(K_X) + aE) = -a(n-1)
\]
and
\[
-1 = l_2 \cdot K_Y = l_2 \cdot (g^*(K_Z) + bE) = -b,
\]
so that $a=0$ and $b=1$. In particular $K_X$ is a Cartier divisor, $-K_X$ is ample and the singular locus of $X$ is the curve $f(E)=f(l_2)$, which is made up by canonical non-terminal singularities.

The ray generated by $[f(l_2)]$ is extremal in $\NE(X)$ and is contracted by a morphism $\varphi: X \rightarrow W$, whose exceptional locus is the curve $f(E)$ and whose flip is $g \circ f^{-1}:X \dashrightarrow Z$.
\end{ex}

Let us recall a theorem concerning birational $K$-negative contractions with fibers of dimension at most one, defined on a varieties which are not necessarily Gorenstein.

\begin{lem}\label{ishii}\cite[Lemma 1.1]{I} Let $X$ be a projective variety with canonical singularities, and let $\varphi: X \rightarrow Y$ be an elementary birational $K_X$-negative contraction whose fibers are at most one-dimensional. If $F_0$ is an irreducible component of a non-trivial fiber of $\varphi$ containing a Gorenstein point of $X$, then $-K_X \cdot F_0 \leq 1$.
\end{lem}

\section{Preliminary results}
\subsection{Mori programs for Fano varieties}\label{ss}
Let us collect in the following theorem some important results about singular Fano varieties. For the definition and the main properties of Mori dream spaces we refer the reader to \cite{hk}.

\begin{thm} \label{teomp}Let $X$ be a $\mathbb{Q}$-factorial Fano variety with canonical singularities. Then for any prime divisor $D \subset X$, there exists a finite sequence
\begin{equation}\label{mp}
X=X_0 \stackrel{\sigma_0}{\dashrightarrow} X_1 \dashrightarrow \cdots \dashrightarrow X_{k-1} \stackrel{\sigma_{k-1}}{\dashrightarrow}X_k \xrightarrow{\psi} Y
\end{equation}
such that, if $D_i \subset X_i$ is the transform of $D$ for $i=1,\ldots,k$ and $D_0:=D$, the following hold:
\begin{enumerate}
 \item $X_1, \ldots, X_k$ and $Y$ are $\mathbb{Q}$-factorial projective varieties and $X_1, \ldots, X_k$ have canonical singularities;
 \item for every $i=0, \ldots, k$ there exists an extremal ray $Q_i$ of $X_i$ with $D_i \cdot Q_i>0$ and $-K_{X_i} \cdot Q_i >0$ such that:
\begin{enumerate}
\item for $i=0, \ldots, k-1$, $\Locus(Q_i) \subsetneq X_i$, and $\sigma_i$ is either the contraction of $Q_i$ (if $Q_i$ is divisorial), or its flip (if $Q_i$ is small);
\item the morphism $\psi:X_k \rightarrow Y$ is the contraction of $Q_k$ and $\psi$ is a fiber type contraction;
\end{enumerate} 
\item $\#\{i \in \{ 0, \ldots, k\} \ | \ Q_i \nsubseteq \N(D_i,X_i)\} = \codim \N(D,X)$.\\Moreover, if we set $c_i:= \codim \N(D_i,X_i)$ for $i=0, \ldots, k$, we have
\[c_{i+1}=\begin{cases} c_{i}\quad&\text{if $Q_i\subseteq\N(D_i,X_i)$}\\
c_{i}-1\quad&\text{if $Q_i\nsubseteq\N(D_i,X_i)$}
\end{cases},\ \text{ and }
c_{k}=\begin{cases} 0\quad&\text{if $Q_k\subseteq\N(D_k,X_k)$}\\
1\quad&\text{if $Q_k\nsubseteq\N(D_k,X_k)$.}
\end{cases}\]
for every $i=0, \ldots, k-1$.
\end{enumerate}
\end{thm}
According to \cite{cas10}, we call a sequence as above a \textit{special Mori program} for the divisor $-D$.

\begin{proof}
Every Fano variety with canonical singularities is a Mori dream space by \cite[Corollary 1.3.2]{bchm}. Hence \cite[Proposition 1.11(1)]{hk}) implies the existence of a Mori Program as in \eqref{mp} where $X_1, \ldots, X_k$ and $Y$ are $\mathbb{Q}$-factorial, and that there are extremal rays $Q_0,\ldots,Q_k$ satisfying (2a), (2b) and such that $D_i \cdot Q_i>0$ for every $i$. The possibility to choose a Mori program in which we also have $-K_{X_i} \cdot Q_i>0$ for $i=1,\ldots,k$, follows from \cite{bchm} as a special case of Mori program with scaling (see also \cite[Proposition 2.4]{cas10}). Then, by \cite[Lemma 3.38]{km}, $X_1,\ldots,X_k$ have canonical singularities. Finally, (3) is proved in \cite[Lemma 2.6(2)(3)]{cas10}.
\end{proof}

In the rest of the paper we will often consider a variety $X$ as follows:
\textit{\begin{equation}\label{X}
\begin{array}{c}
\textmd{$X$ is a $\mathbb{Q}$-factorial, Gorenstein Fano variety of dimension $n$}\\
\textmd{with canonical singularities and with $\dim(NT(X)) \leq 0$}.
\end{array}
\end{equation}}

Notice once for all that \cite[Corollary 5.18]{km} implies that, if a variety $X$ satisfies \eqref{X} and $n \geq 3$, then $\dim(X_{\sing}) \leq n-3$. This fact will be fundamental in the proof of Theorem \ref{teo_princ}.

We are now going to investigate in detail what happens when we run a special Mori program for $-D$ when the Fano variety $X$ satisfies \eqref{X}. We will focus on the extremal rays $Q_i$ of Theorem \ref{teomp}, such that $Q_i \nsubseteq \N(D_i,X_i)$. Using Theorem \ref{teo}, we will prove that, in our setting, they are all divisorial rays, so that the corresponding birational maps $\sigma_i:X_i\dashrightarrow X_{i+1}$ are in fact divisorial contractions. The situation is similar to the smooth case, where such maps are the blow-up of smooth $(n-2)$-dimensional subvarieties (see \cite[Lemma 2.7(1)]{cas10}).

Set $U_0 := X$ and, for every $i=1,\ldots,k$, call $U_i$ the maximal open set of $X_i$ over which the birational map $\sigma_0^{-1} \circ \cdots \circ \sigma_{i-1}^{-1}:X_i \dashrightarrow X$ is an isomorphism. In particular $U_i$ is contained in the Gorenstein locus of $X_i$. Since $D_{i-1} \cdot Q_{i-1} >0$, $X_i \smallsetminus U_i \subseteq D_i$.

With standard arguments (see \cite[Lemma 3.8]{cas09}), the following lemma can be proved:

\begin{lem}\label{ant_deg} Let $X$ be a Fano variety satisfying \eqref{X}. Fix an index $i \in \{0, \ldots, k-1\}$. Let $C \subset X_i$ be an irreducible curve intersecting $U_i$ and $\tilde{C} \subset X$ be its proper transform. Then
\[
 -K_{X_i} \cdot C \geq -K_X \cdot \tilde{C},
\]
and equality holds if and only if $C \subseteq U_i$.
\end{lem}

\begin{lem}\label{teomp2} Let $X$ be a Fano variety satisfying \eqref{X}. In the setting of Theorem \ref{teomp}, set $\{i_1,\ldots,i_s\}:=\{i \in \{0,\ldots,k-1\} \ | \ Q_i \nsubseteq\N(D_i,X_i)\}$. Then:
\begin{itemize}
 \item $s \in \{\codim\N(D,X), \codim\N(D,X)-1\}$. If $s=\codim \N(D,X)$, then $\N(D_k,X_k)=\N(X_k)$; if $s=\codim\N(D,X)-1$, then $\N(D_k,X_k)$ has codimension one in $\N(X_k)$ and $Q_k \nsubseteq \N(D_k,X_k)$;
 \item for every $j=1,\ldots,s$, the map $\sigma_{i_j}$ is a divisorial contraction with fibers of dimension $\leq 1$. Its non-trivial fibers are all irreducible, without multiple one-dimensional components and with reduced structure isomorphic to $\mathbb{P}^1$. Moreover, the general fiber of $\sigma_{i_j}$ is smooth;
 \item for every $j=1,\ldots,s$, let $E_j \subset X$ be the transform of $\Exc(\sigma_{i_j})$. Then $\sigma_{i_j}\circ \cdots \circ \sigma_0:X \dashrightarrow X_{i_j +1}$ is regular (and is an isomorphism) on $E_j$;
 \item $E_1,\ldots,E_s \subset X$ are pairwise disjoint prime divisors;
 \item for every $j=1,\ldots,s$, let $f_j \subset X$ be the transform of a non-trivial fiber of $\sigma_{i_j}$. Then
\[
 [f_j] \notin \N(D,X), \ \ D \cdot f_j>0, \ \ \textmd{and} \ \ E_j \cdot f_j = K_X \cdot f_j=-1.
\]
\end{itemize}
\end{lem}

\begin{proof}
The first assertion follows from Theorem \ref{teomp}(3).

Fix $j \in \{1,\ldots,s\}$ and let $\varphi_{i_j}$ be the contraction of the extremal ray $Q_{i_j}$. There are two possibilities: either $\varphi_{i_j}$ is divisorial and $\varphi_{i_j}=\sigma_{i_j}$, or $\varphi_{i_j}$ is small and $\sigma_{i_j}$ is its flip.

Let $F$ be a non-trivial fiber of $\varphi_{i_j}$ and $F_0$ one of its irreducible components. Since $Q_{i_j} \nsubseteq \N(D_{i_j},X_{i_j})$ and $D_{i_j} \cdot Q_{i_j} >0$, $\varphi_{i_j}$ is finite on $D_{i_j}$. In particular $\dim(F_0)=1$ and $F_0$ intersects $U_{i_j}$, which is made up by Gorenstein points. By Lemmas \ref{ant_deg} and \ref{ishii}, we get:
\[
 1 \leq -K_X \cdot \tilde{F}_0 \leq -K_{X_{i_j}} \cdot F_0 \leq 1,
\]
where $\tilde{F}_0 \subset X$ is the proper transform of $F_0$. Hence $-K_X \cdot \tilde{F}_0 = -K_{X_{i_j}} \cdot F_0 = 1$ and Lemma \ref{ant_deg} assures that $F_0 \subseteq U_{i_j}$. Thus $\Exc(\varphi_{i_j}) \subseteq U_{i_j}$ and $\sigma_{i_j} \circ \cdots \circ \sigma_0$ is regular and is an isomorphism on $E_j$. Moreover, $X_{i_j}$ has canonical singularities and $\dim(NT(X_{i_j}))\leq n-3$ by \cite[Lemma 3.38]{km}. Since $U_{i_j}$ is contained in the Gorenstein locus of $X_{i_j}$ and the fibers of $\varphi_{i_j}$ are at most one-dimensional, we can apply Theorem \ref{teo} and we see that $\varphi_{i_j}$ is a divisorial contraction; in particular it coincides with $\sigma_{i_j}$. All the statements now follow. 
\end{proof}

\subsection{Picard number of divisors in Fano varieties}

Given a Fano variety $X$, possibly singular, it makes sense to consider the following invariant:
\[
 c_X=\max\{\codim \N(D,X) \ | \ D \textmd{ is a prime divisor in } X\}.
\]
This invariant was introduced in \cite{cas10} in the smooth case, where the author proved that it is always $\leq 8$ (see Theorem \ref{cinzia}). We will use the same invariant to prove Theorem \ref{teo_princ}.

In Proposition \ref{prop} we study what happens when we run a special Mori program for $-D$ when $c_X \geq 4$ and $D$ is a prime divisor with $\codim \N(D,X)=c_X$.

Let us first recall a preliminary result concerning projective bundles. We need to state it in the analytic setting; for simplicity, if $X$ is an algebraic variety, we still denote by $X$ the corresponding analytic variety.

\begin{rmk}\label{trivial}
Let $X$ and $Y$ be analytic varieties. Let $p:X \rightarrow Y$ be a holomorphic $\mathbb{P}^n$-bundle. Suppose that there exist $n+2$ sections $s_0, \ldots, s_{n+1}$ of $p$ such that $s_0(y), \ldots, s_{n+1}(y)$ are projectively indipendent for every $y \in Y$. Then $p$ is the trivial $\mathbb{P}^n$-bundle over $Y$ and for every $i=0,\ldots,n+1$, $s_i(Y)$ is of the type $\{pt\}\times Y$.

This is an elementary fact, for which we could not find a reference. The point is that $s_0,\ldots,s_{n+1}$ can be locally lifted to holomorphic functions to $\mathbb{C}^{n+1}$ which are uniquely determined up to a common holomorphic multiple, as basic linear algebra arguments show. In terms of the vector bundle corresponding to $p$, this means that any two trivializations differ, in the intersections, by the multiplication with a non-vanishing holomorphic function. Thus the images of these trivializations in the projective space glue together giving a global holomorphic trivialization for $p$.
\end{rmk}

\begin{pro}\label{prop} Let $X$ be a Fano variety which satisfies  \eqref{X} and such that $c_X \geq 4$. Let $D$ be a prime divisor of $X$ with $\codim\N(D,X)=c_X$ and let $E_1, \ldots, E_s$, $f_1,\ldots,f_s$ be as in Lemma \ref{teomp2}. Then the vector space
\[
 L:=\N(D \cap E_i,X) \subseteq \N(X)
\]
has codimension $c_X+1$ and does not depend on $i \in \{i,\ldots,s\}$. Moreover for every $i = 1, \ldots,s$:
\begin{enumerate}
 \item[(1)] $L=\N(D,X) \cap E_i^{\perp}=\N(E_i,X) \cap E_i^{\perp}$;
 \item[(2)] $\codim\N(E_i,X)=c_X$;
 \item[(3)] $\N(E_i,X) = \N(D \cap E_i,X) \oplus \mathbb{R}[f_i]$;
 \item[(4)] $R_i:=\mathbb{R}_{\geq 0}[f_i]$ is an extremal ray of type $(n-1,n-2)^{eq}$ and, if $\varphi_i: X \rightarrow Y_i$ is its contraction, then $Y_i$ is Fano.
\end{enumerate}
If furthermore there exists an extremal ray $R_0$ of type $(n-1,n-2)^{eq}$ such that $D=\Locus(R_0)$ and $E_i \cdot R_0 > 0$ for every $i=1,\ldots,s$, then:
\begin{enumerate}
 \item[(5)] there exists a special Mori program for $-E_i$ such that $D$ is one of the prime divisors it determines (in the sense of Lemma \ref{teomp2});
 \item[(6)] set $E_0:=D$. For every $i =0,\ldots,s$, there exists an $(n-2)$-dimensional variety $T_i$ and finite morphisms $T_i \rightarrow \varphi_i(E_i)$ and $h_i: \mathbb{P}^1 \times T_i \rightarrow E_i$ making the following diagram commute:
\[
\xymatrix{
\mathbb{P}^1 \times T_i \ar[r]^{\hspace{.3cm}h_i} \ar[d] & E_i \ar[d]^{\varphi_{i|E_i}}\\
T_i \ar[r] & \varphi_i(E_i)
}
\]
where $\mathbb{P}^1 \times T_i \rightarrow T_i$ is the projection map. Moreover
\[
 (j_i \circ h_i)_{*} \N(\{pt\} \times T_i) = L,
\]
where $j_i:E_i \hookrightarrow X$ is the inclusion map;
 \item[(7)] for every $i=0,\ldots,s$, $R_i$ is the unique extremal ray of $X$ having negative intersection with $E_i$.
\end{enumerate}

\end{pro}

\begin{proof}
Let 
\[
X=X_0 \stackrel{\sigma_0}{\dashrightarrow} X_1 \dashrightarrow \cdots \dashrightarrow X_{k-1} \stackrel{\sigma_{k-1}}{\dashrightarrow}X_k \xrightarrow{\psi} Y
\]
be the special Mori program for $-D$ giving rise to the prime divisors $E_1,\ldots,E_s$. By Lemma \ref{teomp2}, we know that $s \in \{c_X,c_X-1\}$; in particular $s \geq 3$ and we can take three distinct indices $i,j,l \in \{1,\ldots,s\}$. Let $m_i \in \{0,\ldots,k-1\}$ be such that $E_i \subset X$ is the transform of $\Exc(\sigma_{m_i})$. Since $D \cdot f_i >0$, $D \cap E_i$ dominates the image of $E_i$ under the composition $\sigma_{m_i} \circ \cdots \circ \sigma_0:X \dashrightarrow X_{m_i+1}$, which is regular on $E_i$ (Lemma \ref{teomp2}). Then
\begin{equation}\label{-}
 \N(E_i,X)=\N(D \cap E_i,X) + \mathbb{R}[f_i].
\end{equation}
Recall that $E_i \cap E_j = \emptyset$, so that $\N(D \cap E_i,X) \subseteq E_j^{\perp}$. Thus
\begin{equation}\label{--}
 c_X + 1 \geq \codim\N(E_i,X) + 1\geq \codim\N(D \cap E_i,X) \geq \codim(\N(D,X) \cap E_j^{\perp})=c_X+1, 
\end{equation}
where the first inequality follows from the definition of $c_X$ and the last equality holds because $D$ intersects $E_j$ and hence $\N(D,X) \nsubseteq E_j^{\perp}$. In particular from \eqref{--} we have $\N(D \cap E_i,X)=\N(D,X) \cap E_j^{\perp}$. Repeating the same reasoning for the pair of disjoint prime divisors $E_j$ and $E_l$, we get:
\begin{equation}\label{---}
 \N(D \cap E_i,X)=\N(D,X) \cap E_j^{\perp}= \N(D \cap E_l,X).
\end{equation}
Statements (1), (2) and (3) follow from \eqref{-}, \eqref{--} and \eqref{---} letting $i$, $j$ and $l$ vary in $\{1,\ldots,s\}$. Just notice that the last equality in (1) holds because $L = \N(D \cap E_i,X)=\N(D,X) \cap E_i^{\perp} \subseteq \N(E_i,X) \cap E_i^{\perp}$; but $E_i \cdot f_i <0$, so that $\N(E_i,X) \nsubseteq E_i^{\perp}$, and hence $L=\N(E_i,X) \cap E_i^{\perp}$ for dimensional reasons.
\medskip

In order to prove (4), let us first show that $-K_X + E_i$ is a nef divisor and $(-K_X + E_i)^{\perp} \cap \NE(X)=R_i$, so that $R_i$ is an extremal ray. To see this, let $C \subset X$ be an irreducible curve. If $C \nsubseteq E_i$, then clearly $(-K_X + E_i) >0$; the same holds if $C \subseteq D \cap E_i$, because we have just proved that $\N(D \cap E_i,X) \subseteq E_i^{\perp}$. Assume now that $C \subseteq E_i$. By \cite[Lemma 3.2 and Remark 3.3]{occ06}, we have $C\equiv \lambda f_i + \mu C'$, for a curve $C'\subseteq D \cap E_i$ and real coefficients $\lambda$ and $\mu$, with $\mu \geq 0$. Since $(-K_X + E_i) \cdot f_i=0$, we have:
\[
(-K_X+E_i) \cdot C = \mu (-K_X + E_i) \cdot C' \geq 0,
\]
and equality holds if and only if $[C] \in R_i$.

Let $\varphi_i:X \rightarrow Y_i$ be the contraction of the extremal ray $R_i$; it is clear that $\Exc(\varphi_i)=E_i$. Moreover, by Lemma \ref{teomp2}, $\varphi_i$ is of type $(n-1,n-2)^{eq}$; in particular, by Theorem \ref{teo}(4),
\[
 K_X \sim \varphi^{*}(K_{Y_i}) + E_i,
\]
hence $-K_{Y_i}$ has a multiple which is Cartier and ample.
\medskip

Let us now suppose $D=\Locus(R_0)$ with $R_0$ an extremal ray of type $(n-1,n-2)^{eq}$. Let $\varphi_0:X \rightarrow Y$ be the corresponding contraction. Since $E_i \neq D$, $\varphi_0(E_i) \subset Y$ is a prime divisor. Being $Y$ Fano, by Lemma \ref{teomp}, there exists a special Mori program for $-\varphi_0(E_i) \subset Y$. Together with $\varphi_0$, this gives a special Mori program for $-E_i$ where the first extremal ray is $R_0$:
\[
 X \xrightarrow{\varphi_0} Y \dashrightarrow Z_1 \dashrightarrow \cdots \dashrightarrow Z_{k-1} \dashrightarrow Z_k.
\]
Notice that $R_0 \nsubseteq \N(E_i,X)$, otherwise, by (1), it would belong to $\N(D \cap E_i,X) \subset E_i^{\perp}$ and this is impossible since $E_i \cdot R_0 >0$ by assumption. Statement (5) is thus proved.
\medskip

The proof of (6) requires some work. Define $S_i:=\varphi_i(E_i) \subset Y_i$ for $i=0,\ldots,s$. Let us first show that for every $i=0,\ldots,s$, there exist three pairwise disjoint $(n-2)$-dimensional subvarieties $F_{i}^{1}$, $F_{i}^2$, $F_{i}^3 \subset E_i$ such that $\varphi_{i|F_{i}^{j}}: F_{i}^{j} \rightarrow S_i$ is finite for $j=1,2,3$. Let us examine separately the cases $i=0$ and $i \in \{1,\ldots,s\}$. Suppose first $i=0$ and define
\[
 F_{0}^1=E_0 \cap E_1, \ \ \ \ F_{0}^2=E_0 \cap E_2, \ \ \ \ F_{0}^3=E_0 \cap E_3.
\]
Since for every $j=1,2,3$, $E_j \cdot R_0>0$, we have $R_0 \nsubseteq \N(E_0 \cap E_j,X)$. Hence the restriction of $\varphi_0$ to $F_{0}^j$ is finite.

Suppose now $i \in \{1,\ldots,s\}$. Running a special Mori program for $-E_i$, we get, as in Lemma \ref{teomp2}, $s_i$ pairwise disjoint divisors $E_{i}^{1}, \ldots, E_{i}^{s_i}$. Since $\codim \N(E_i,X)=c_X\geq 4$ by (2), we have $s_i \geq 3$. According to (5), we can suppose $E_i^1=E_0$. Moreover, since $E_0 \cdot R_i >0$, we get $E_i^j \cdot R_i>0$ for every $j=2,\ldots,s_i$. Indeed, if $E_i^j \cdot R_i=0$ for some $j$, there must be a curve $f_i$ such that $[f_i] \in R_i$ and $f_i \subset E_i^j$ (because $E_i \cap E_i^j \neq \emptyset$). But then $E_0 \cdot f_i=0$, because $E_0 \cap E_i^j = \emptyset$.

Define
\[
 F^1_i:=E_i \cap E_0, \ \ \ \ F^2_i:=E_i \cap E_i^2, \ \ \ \ F^3_i:=E_i \cap E_i^3.
\]
As above, since $E_0 \cdot R_i$, $E_i^2 \cdot R_i$ and $E_i^3 \cdot R_i$ are all positive, we see that the restrictions of $\varphi_i$ to $F^1_i$, $F^2_i$ and $F^3_i$ are finite.

Notice that, in both the cases $i=0$ and $i>0$, $F^{1}_i$ is of the form $E_0 \cap E_l$ for some $l \in \{1,\ldots,s\}$.
\medskip

Fix $i\in\{0,\ldots,s\}$ and consider the normalization $\nu: \tilde{S} \rightarrow S_i$ of $S_i$. Let us first prove that there exists a holomorphic $\mathbb{P}^1$-bundle $p:\tilde{E} \to \tilde{S}$ and a finite morphism $\mu:\tilde{E} \to E_i$ making the following diagram commute
\[
 \xymatrix{
\tilde{E} \ar[r]^{\mu}  \ar[d]^{p}& E_i \ar[d]^{\varphi_i}\\
\tilde{S} \ar[r]^{\nu} & S_i
}
\]
where, for simplicity, we still denote with $\varphi_i$ its restriction to $E_i$.

Fix a smooth non-trivial fiber $f_i \simeq \mathbb{P}^1$ of $\varphi_i$ and call $\hat{S}$ the connected component of $\Hilb(X)$ containing the point which corresponds to $f_i$. Then, by Lemma \ref{lem_ando}, $\dim(\hat{S})=n-2$ and there is a birational map $\xi:S_i \dashrightarrow \hat{S}$. Let $\hat{S} \times X \supset \hat{E}\xrightarrow{\alpha} \hat{S}$ be the restriction to $\hat{S}$ of the universal family over $\Hilb(X)$, and call $\beta: \hat{E} \rightarrow X$ the other projection. Since cycles in the same connected component of $\Hilb(X)$ are numerically equivalent, every point of $\hat{S}$ represents a scheme contracted by $\varphi_i$ and $\beta(\hat{E})=E_i$. We have the diagram:

\begin{equation}\label{diag}
\xymatrix{
\hat{E} \ar[r]^{\beta} \ar[d]^{\alpha} & E_i \ar[d]^{\varphi_i}\\
\hat{S} & S_i \ar@{-->}[l]_{\xi}.
}
\end{equation}
Choose a point $z \in \hat{S}$; let $Z \subset X$ be the closed subscheme such that $z=[Z]$ and denote by $\Gamma$ the corresponding one-cycle. By Theorem \ref{teo}(1), we know that $\Gamma \simeq \mathbb{P}^1$ and that, if we look at $\Gamma$ as a subscheme of $Z$, its ideal sheaf $\mathcal{I}_{\Gamma}$ is a skyscraper sheaf. We can write:
\[
 \Gamma = \Spec \frac{\mathcal{O}_{Z}}{\mathcal{I}_{\Gamma}}.
\]
By the flatness of $\alpha$, we have: $\chi(\mathcal{O}_{Z})=\chi(\mathcal{O}_{f_i})=1$. Therefore $\chi(\mathcal{I}_{\Gamma})=0$, which implies $\mathcal{I}_{\Gamma}=0$ and $\Gamma \simeq Z$. We have thus shown that every fiber of $\alpha$ is isomorphic to $\mathbb{P}^1$.

Consider now the rational map $\psi:=\xi \circ \nu: \tilde{S} \dashrightarrow \hat{S}$ and let us prove that it is in fact a morphism. Fix a point $y \in \tilde{S}$ such that $\xi$ is not defined at $\nu(y)$. Let $C \subset \tilde{S}$ be a curve passing through $y$ and such that $\nu(C \smallsetminus \{y\})$ intersects the domain of $\xi$. Eventually composing with its normalization, we can suppose $C$ to be smooth, so that the restriction of $\psi$ to $C$ can be extended to $y$. By the commutativity of \eqref{diag}, we see that the only possibility is $\psi_{|C}(y)=[(\varphi_i^{-1}(\nu(y)))_{\red}]$. In particular this point does not depend on the curve $C$ and is the total transform of $y$ through $\psi$. Then, by Zariski main theorem, $\psi$ is regular at $y$.

We can finally set $\tilde{E}:=\tilde{S} \times_{\hat{S}} \hat{E}$ and we get
\[
 \xymatrix{
\tilde{E} \ar@/^1pc/[rr]^{\mu}  \ar[d]^{p} \ar[r] & \hat{E} \ar[r]_{\beta} \ar[d]^{\alpha} & E_i \ar[d]^{\varphi_i}\\
\tilde{S} \ar[r]^{\psi} \ar@/_1pc/[rr]_{\nu} & \hat{S} & S_i \ar@{-->}[l]_{\xi}
}
\]
where $p:\tilde{E} \rightarrow \tilde{S}$ is flat. Moreover the fibers over every (closed) points of $\tilde{S}$ are isomorphic to $\mathbb{P}^1$. By \cite[12.1.6]{ega2}, we see that also the fibers over the non-closed points of $\tilde{S}$ are smooth rational curves; then \cite[6.8.3]{ega1} shows that $\tilde{E}$ is normal. We can now apply \cite[Theorem II.2.8]{kol_rat}, and conclude that $p$ is a holomorphic $\mathbb{P}^1$-bundle.

Notice that $\mu: \tilde{E} \to E_i$ is finite and birational; in particular $\tilde{E}$ is the normalization of $E_i$.
\medskip

For $j=1,2,3$, let $Z^j \subseteq \tilde{E}$ be an irreducible component of $\mu^{-1}(F_i^j)$ such that $\mu(Z^j)$ still dominates $S_i$. Consider the following pull-back diagram:
\[
\xymatrix{
G^1 \ar[r]^{h^1} \ar[d]_{\alpha^1} & \tilde{E} \ar[d]^{p}\\
Z^1 \ar[r]_{p} & \tilde{S},
}
\]
where, in the lower horizontal arrow, we still write $p$ for its restriction to $Z^1$.
Then, by the universal property of fiber product, there exists a section $s^1$ of $\alpha^1$ such that $h^1 \circ s^1$ is the inclusion of $Z^1$ in $\tilde{E}$. Consider now the restriction of $\alpha^1$ to an irreducible component of $(h^1)^{-1}(Z^2)$. It is still surjective; let us consider the base change of $\alpha^1$ given by this map and call $\alpha^2$ the resulting morphism. As above, call $s^2$ the natural section of $\alpha^2$ and call $(s^{1})^*$ the section pull-back of $s^1$; note that the images of $s^2$ and $(s^{1})^*$ are disjoint. Repeating this reasoning once again with $Z^3$ and composing with $\nu$, we get a finite map $T_i \rightarrow S_i$ giving a holomorphic $\mathbb{P}^1$-bundle $G^3 \to T_i$ which has three disjoint sections.

We can now apply Remark \ref{trivial} and conclude that $G^3 \to T_i$ is the trivial holomorphic $\mathbb{P}^1$-bundle over $T_i$, \textit{i.e.} there exists a biholomorphic map $G^3 \to T_i \times \mathbb{P}^1$ commuting with the projection map into $T_i$; moreover the images of the three disjoint sections through this biholomorphism are all of the type $\{pt\}\times T_i$. Being $G^3$ and $T_i$ projective varieties, this biholomorphic map is an isomorphism of algebraic varieties. 
\[
\xymatrix{
\mathbb{P}^1 \times T_i = G^3 \ar@/^1.5pc/[rrrr]^{h_i} \ar[r]^{\hspace{.7cm}h^3} \ar[d] & G^2  \ar[r]^{h^2} \ar[d]_{\alpha^2} & G^1 \ar[r]^{h^1} \ar[d]_{\alpha^1} & \tilde{E} \ar[r] \ar[d]^{\alpha} & E_i \ar[d]^{\varphi_i}\\
T_i = (h^1 \circ h^2)^{-1}(Z^3) \ar@/_/[u]_{s^3} \ar[r]^{\hspace{.7cm}\alpha^2} & (h^{1})^{-1}(Z^2) \ar@/_/[u]_{s^2} \ar[r]^{\hspace{.4cm}\alpha^1} & Z^1 \ar@/_/[u]_{s^1} \ar[r]^{p} & \tilde{S} \ar[r]^{\nu} & S_i.
}
\]
Let $(s^1)^{**}$ be the section of $\mathbb{P}^1 \times T_i$ obtained by pulling back $s_i^1$. By construction there is a finite morphism
\[
 h_{i|(s^1)^{**}(T_i)}: \{pt\} \times T_i \rightarrow \mu(Z^1) \subseteq F_i^1.
\]
Remember that there exists an index $l \in \{1, \ldots, s\}$ such that $F_i^1=E_0 \cap E_l$, so that
\[
(j_i \circ h_i)_*\N(\{pt\}\times T_i) \subseteq \N(F^1_i,X)=L.
\]
Furthermore, since $\mu(Z^1)$ dominates $S_i$, we have
\[
\dim (j_i \circ h_i)_*\N(\{pt\}\times T_i)= \dim\N(\mu(Z^1),X)\geq \dim\N(S_i,Y_i)=\rho_X-c_X-1.
\]
Hence, for dimensional reasons, $(j_i \circ h_i)_*\N(\{pt\}\times T_i)=L$ and (6) is finally proved.
\medskip

To prove (7), take an extremal ray $R$ of $X$ such that $E_i \cdot R<0$. Then $R \subseteq \overline{\NE}(E_i,X) \subseteq \NE(X)$, and hence $R$ is an extremal ray of $\overline{\NE}(E_i,X)$. From (6) we have:
\begin{equation}\label{incl}
\begin{array}{ll}
 \overline{\NE}(E_i,X) & = (j_i \circ h_i)_* \overline{\NE}(\mathbb{P}^1 \times T_i) = \\
 & = (j_i \circ h_i)_* \overline{\NE}(\{pt\} \times T_i) + (j_i \circ h_i)_*\overline{\NE}(\mathbb{P}^1 \times \{pt\}) \subseteq \\
 & \subseteq (L \cap \NE(X)) + R_i.
\end{array}
\end{equation}
Being $R$ extremal, it will be $R \subseteq L \cap \NE(X)$ or $R=R_i$. Since $L \subseteq E_i^{\perp}$ and $E_i \cdot R <0$, it must be $R=R_i$. The proposition is thus proved.
\end{proof}

\begin{rmk}\label{interi}
In the setting of Proposition \ref{prop}, for every $i=0,\ldots,s$, the general non-trivial fiber of $\varphi_i$ is contained in $X_{\reg}$. Thus the intersection products $E_i \cdot f_0$ and $E_0 \cdot f_i$ are integral numbers for every $i=1,\ldots,s$.
\end{rmk}

\section{Proof of Theorem \ref{teo_princ}}

Let us finally prove Theorem \ref{teo_princ}. The idea is the same as in the smooth case (\cite[Proposition 3.2.1]{cas10}). Nevertheless, for the reader's convenience, we write  here almost all of the details and we refer to the cited paper only for few results whose proofs' lack of knowledge does not affect the understanding of the rest of the proof.

The proof is quite articulated and will cover the whole section. Let us give a short outline. If $n=2$, the theorem is well-known; we may thus suppose $n \geq 3$. Let us notice, moreover, that the theorem is proved if we verify its statements under the assumption $c_X \geq 4$.

In the first part of the proof, we use the preliminary results of the second section in order to find a \textquotedblleft suitable \textquotedblright divisor to which we apply Proposition \ref{prop}. The Mori program we obtain allows us to define a contraction $\psi:X \to Y$, whose general fiber is a Del Pezzo surface with Picard number $c_X + 1$. Moreover this surface is smooth; the fundamental fact here is that $\dim(X_{\sing}) \leq n-3$. Thus $c_X \leq 8$ and the first part of the theorem is proved. What is left to show at this point is the existence of another contraction $\xi:X \to S$ giving rise to a finite morphism $\pi:=(\xi,\psi):X \to S \times Y$ as in the theorem. This construction will require some more work.

\begin{proof}[Proof of Theorem \ref{teo_princ}]
Let us suppose $n \geq 3$ and $c_X \geq 4$, so that all the assumptions of the theorem hold. If there exists a finite morphism $\pi:X \to S \times Y$ with $9 \geq \rho_S \geq c_X +1$, then $\rho_X - \rho_D \leq \codim \N(D,X) \leq c_X \leq 8$ for every prime divisor $D$. Thus it is enough to prove the second statement.

Let us notice that Proposition \ref{prop} implies at once the existence of an extremal ray $R_0$ of type $(n-1,n-2)^{eq}$ such that the target $Y_0$ of its contraction is Fano and, if $E_0:=\Locus(R_0)$, then $\N(E_0,X)=c_X$. In fact, it is enough to take one of the prime divisors obtained, as in Lemma \ref{teomp2}, from a special Mori program for $-D$ when $D$ is a divisor with $\codim\N(D,X)=c_X$. Let us fix such an extremal ray and consider a special Mori program for $-E_0$
\[
 X=X_0 \stackrel{\sigma_0}{\dashrightarrow} X_1 \dashrightarrow \cdots \dashrightarrow X_{k-1} \stackrel{\sigma_{k-1}}{\dashrightarrow}X_k \xrightarrow{\psi} Y;
\]
let $E_1,\ldots,E_s$ be the prime divisors it determines, in the sense of Lemma \ref{teomp2}.

As in the proof of Proposition \ref{prop}, we see that one of the following possibilities must hold: either $E_1 \cdot R_0 = \cdots = E_s \cdot R_0 =0$, or $E_i \cdot R_0 >0$ for every $i=1,\ldots,s$. By \cite[Lemma 3.2.10]{cas10}, we can always suppose to be in the second case (notice that, though such a result is stated in the smooth case, everything works also in our setting). Thus the assumptions of Proposition \ref{prop}(5) are verified and we can, at the occurrence, look at $E_0$ as one of the divisors determined by a special Mori program for $-E_i$. Hence all the claims of Proposition \ref{prop} hold if we interchange the roles of $E_0$ and $E_i$; in particular we get $L \subset E_0^{\perp}$.
\medskip

Let us consider the divisor $-K_X + E_1 + \cdots+ E_s$ on $X$. By Proposition \ref{prop}(7), for every extremal ray $R$ of $X$, $(-K_X + E_1 + \cdots + E_s) \cdot R \geq 0$. Moreover equality holds if and only if $R=R_i$ for some $i \in \{1,\ldots,s\}$. Thus $-K_X + E_1 + \cdots + E_s$ is a nef divisor and it defines a contraction $\sigma: X \rightarrow X_s$ such that ker$(\sigma_*)=\mathbb{R}R_1 + \cdots + \mathbb{R}R_s$ and $\Exc(\sigma)=E_1 \cup \cdots \cup E_s$. In particular $\dim(\ker \sigma_*)=s$. Notice that $\sigma$ verifies the assumptions of Theorem \ref{teo}, so that
\[
 -K_X + E_1 + \cdots + E_s = \sigma^{*}(-K_{X_s})
\]
and $X_s$ is Fano.

Set $D_0:=\sigma(E_0) \subset X_s$. Since $[f_i] \notin \N(E_0,X)$ (otherwise it would belong to $\N(E_0 \cap E_i,X) \subseteq E_i^{\perp}$ by Proposition \ref{prop}(1)), $\sigma_{|E_0}: E_0 \rightarrow D_0$ is a finite morphism. Hence $D_0 \subset X_s$ is a divisor and, by Proposition \ref{prop}(3),
\begin{equation}\label{eq}
 \N(D_0,X_s) = \sigma_*(L) + \mathbb{R}[\sigma(f_0)].
\end{equation}
Notice that $\sigma^{*}D_0=E_0 + \sum_{i=1}^s (E_0 \cdot f_i) E_i$. Thus, using the projection formula and recalling that the intersection products $E_i \cdot f_0$ and $E_0 \cdot f_i$ are all integers (Remark \ref{interi}), we see that $D_0 \cdot \sigma(f_0) > 0$ and $\sigma_*(L) \subseteq D_0^{\perp}$.

Factoring $\sigma$ as a sequence of $s$ divisorial contractions, we can view $\sigma:X \rightarrow X_s$ as a part of a special Mori program for $-E_0$ with $s$ steps such that at each step we have $Q_i \nsubseteq \N((E_0)_i,X_i)$. Recall that there are two possibilities: either $s=c_X$ and $\N(D_0,X)= \N(X_s)$, or $s=c_X - 1$ and $\codim\N(D_0,X_s)=1$.
\medskip

Let us now show that, up to replacing $s$ with $s+1$, we can assume that there exists an elementary contraction of fiber type $\varphi:X_s \rightarrow Y$ such that $D_0 \cdot \NE(\varphi)>0$.

Let $R$ be an extremal ray of $X_s$ such that $D_0 \cdot R >0$ and call $\varphi_R$ the contraction it defines. If $\varphi_R$ is of fiber type we are done. Suppose that it is birational; then it is enough to show that 
$R \nsubseteq \N(D_0,X_s)$. In fact, if this is true, we can view the contraction $\varphi_R:X_s \to X_{s+1}$ as a part of a special Mori program for $-E_0$ with $s+1$ steps. In particular it must be $s+1=c_X$ and $\N(\varphi_R (D_0),X_{s+1})=\N(X_{s+1})$ (see Lemma \ref{teomp2}). We can now replace $X_s$ with $X_{s+1}=X_{c_X}$; given now an extremal ray $R'$ of $X_{c_X}$ with $\varphi_R(D_0) \cdot R'>0$, it will necessarily be $R' \subseteq \N(\varphi_R (D_0),X_{c_X})$. Thus the above argument shows that the contraction $\varphi_{R'}$ defined by $R'$ cannot be birational anymore, and we are done.. 

In order to prove that $R \nsubseteq \N(D_0,X_s)$, let us first show that $R \nsubseteq \overline{\NE}(D_0,X_s)$. If, by contradiction, this is the case, then $R$ is a one-dimensional face of $\overline{\NE}(D_0,X_s)$. By \eqref{incl}, we have $\overline{\NE}(E_0,X) \subseteq R_0 + (L \cap \NE(X))$, and then
\[
 \overline{\NE}(D_0,X_s) \subseteq \sigma_*R_0 + (\sigma_*(L \cap \NE(X))).
\]
Hence $R = \sigma_*(R_0)$ because $\sigma_*(L) \subseteq D_0^{\perp}$. Then $D_0 \subseteq \Locus(R)$, which is impossible, because $D_0 \cdot R >0$. Thus $R \nsubseteq \overline{\NE}(D_0,X_s)$ and $\varphi_R$ is finite on $D_0$; in particular the fibers of $\varphi_R$ are at most one-dimensional. For $i=1,\ldots,s$, set $G_i:=\sigma(E_i) \subset X_s$; then $\dim(G_i)=n-2$, $\N(G_i,X_s)=\sigma_*(L)$ and $G_i \subset D_0$. Let $C$ be an irreducible component of a one-dimensional fiber of $\varphi_R$. Then $C$ cannot be contained in $G_1 \cup \cdots \cup G_s \subset D_0$, and hence it intersects the open subset over which $X_s$ is isomorphic to $X$, which is Gorenstein. Applying now Lemma \ref{ishii}, we see that $-K_X \cdot C\leq 1$. By Lemma \ref{ant_deg}:
\[
 1 \leq -K_X \cdot \tilde{C} \leq -K_X \cdot C \leq 1, 
\]
where $\tilde{C} \subset X$ is the strict transform of $C$. Then $-K_X \cdot \tilde{C} = -K_X \cdot C = 1$ and $C \cap G_1 \cap \cdots \cap G_s = \emptyset$. Thus the exceptional locus of $\varphi_R$ is contained in the Gorenstein locus of $X_s$; using now Theorem \ref{teo}(2), we see that $\varphi_R$ is divisorial of type $(n-1,n-2)^{eq}$. Let $E_R$ be its exceptional divisor; the above argument shows that $\N(E_R,X_{s}) \subseteq (G_1)^{\perp} \cap \cdots \cap (G_s)^{\perp}$. In particular $\sigma_*(L)=\N(G_1,X_s)\subseteq E_R^{\perp}$.

Let us now show that $R \nsubseteq \N(D_0,X_s)$. Otherwise, by \eqref{eq}, if $C$ is an irreducible curve with class in $R$ as above, it would be $C\equiv\lambda\sigma(f_0) + \sigma_*(\eta)$ with $\eta \in L$. Recalling that $\sigma_*(L)\subseteq (D_0)^{\perp}$, we get $0 < D_0 \cdot C = \lambda D_0 \cdot \sigma(f_0)$; then $\lambda >0$ because $D_0 \cdot \sigma(f_0)>0$. But $E_R \neq D_0$, hence
\[
 0>E_R \cdot C =\lambda E_R \cdot \sigma(f_0) \geq 0,
\]
and we get a contradiction.

\medskip

Let $\varphi:X_s \rightarrow Y$ be the contraction of fiber type whose existence we have just proved. Then $D_0 \cdot \NE(\varphi)>0$ and, if we set $\psi=\varphi \circ \sigma:X \rightarrow Y$, we have $\psi(E_0)=Y$.
\[
 \xymatrix{
X \ar[r]_{\sigma} \ar@/^1pc/[rr]^{\psi} & X_s \ar[r]_{\varphi} & Y.
}
\]
Since $\N(G_i,X_s)=\sigma_*(L) \subseteq D_0^{\perp}$, $\varphi$ must be finite on $G_i$, so that $\dim(Y)\geq n-2$. Moreover $\rho_Y=\rho_X-s-1$.

\medskip
\noindent\textbf{First case: $\varphi$ is not finite on $D_0$.} In this case $\NE(\varphi) \subseteq \N(D_0,X_s)$, hence $s=c_X$ by Theorem \ref{teomp}(4).

Simple computations show that $E_0,\ldots,E_{c_X}$ and $\mathbb{R}[f_0],\ldots,\mathbb{R}[f_{c_X}]$ are linearly indipendent in $\mathcal{N}^1(X)$ and $\N(X)$ respectively. In particular, recalling that $L \subseteq E_0^{\perp} \cap E_1^{\perp} \cap \cdots \cap E_{c_X}^{\perp}$ and that $\codim(L)=c_X+1$, we have
\[
L = E_0^{\perp} \cap E_1^{\perp} \cap \cdots \cap E_{c_X}^{\perp}.
\]

Since $R=\NE(\varphi)$ is a one-dimensional face of $\overline{\NE}(D_0,X_s)\subseteq\sigma_*R_0 + (\sigma_*(L \cap \NE(X)))$, it must be $R=\sigma_*(R_0)$ (recall that $\sigma_*(L) \subseteq D_0^{\perp}$). In particular $\dim(Y)= \dim(\varphi(D_0)) =n-2$.

Let $F$ be the general fiber of $\psi$. By construction $\N(F,X)=\mathbb{R}[f_0] + \cdots + \mathbb{R}[f_{c_X}]$. Moreover $F \subset X_{\reg}$ because $\dim(X_{sing}) \leq n-3$. Then $F$ is a smooth Del Pezzo surface and
\[
9 \geq \rho_F \geq \dim(\N(F,X))=c_X+1.                                                                                                                                                                                                                                                                                                                                                                                               \]
Then $c_X \leq 8$ and the first statement of Theorem \ref{teo_princ} is proved. 

Let us now construct the finite morphism $\pi$. Let us consider the divisor
\[
 M:=2E_0 + \sum_{i=1}^{c_X}E_i
\]
on $X$ and let us verify that it is nef. If $C \subset \Supp(M)$ is an irreducible curve, than $C \subset E_j$ for some $j \in \{1,\ldots,c_X\}$. By \eqref{incl}, $[C] \in L + R_j$. Since $L = E_0^{\perp} \cap E_1^{\perp} \cap \cdots \cap E_{c_X}^{\perp} \subseteq M^{\perp}$, it is enough to compute the intersection products $M \cdot f_j$. By Lemma \ref{teomp2}:
\[
M \cdot f_0 = -2 + \sum ^{c_X}_{i=1} E_i \cdot f_0 \ \ \textmd{ and }\ \ M \cdot f_i = 2 E_0 \cdot f_i -1  \ \ \textmd{for $i=1,\ldots,c_X$}.
\]
Recall that $E_i \cdot f_0$ and $E_0 \cdot f_i$ are all positive numbers and that, by Remark \ref{interi}, they are integral. Moreover $c_X\geq 4$, so that all the above intersection products are positive. Hence $M$ is a nef divisor and it defines a contraction $\xi:X \rightarrow S$ such that $\NE(\xi)=M^{\perp} \cap \NE(X)$.
\[
 \xymatrix{ 
& X \ar[r]^{\sigma} \ar[rd]^{\psi} \ar[dl]_{\xi}& X_{c_X} \ar[d]^{\varphi}\\
S & & Y
}
\]
\medskip

For every $i=1,\ldots,c_X$, let $h_i:\mathbb{P}^1 \times T_i \rightarrow E_i$ be the finite morphism given by Proposition \ref{prop}(6) and let $\mathbb{P}^1 \times T_i \xrightarrow{\gamma_i} Z_i \xrightarrow{\delta_i} \xi(E_i)$ be the Stein factorization of $(\xi_{|E_i}) \circ h_i$, so that $\gamma_i$ has connected fibers and $\delta_i$ is finite. Since $(j_i \circ h_i)_*\N(\{pt\} \times T_i)=L=E_0^{\perp} \cap \cdots \cap E_{c_X}^{\perp} \subseteq \ker(\xi_*)$, $\gamma_i(\{pt\} \times T_i)$ is contracted to a point by $\gamma_i$. Then $\gamma_i$ factors through the projection $\mathbb{P}^1 \times T_i \rightarrow \mathbb{P}^1$:
\[
 \xymatrix{
\mathbb{P}^1 \ar[r] & Z_i \ar[rd]^{\delta_i} & \\
\mathbb{P}^1 \times T_i \ar[u] \ar[ur]^{\gamma_i} \ar[r]^{\hspace{.35cm}h_i}& E_i \ar[r]^{\hspace{-.3cm}\xi_{|E_i}} & \xi(E_i). 
}
\]
In particular $\xi(E_i)=\xi(f_i)$ is an irreducible curve because $M \cdot f_i>0$.

Let us show that $\NE(\xi)=L \cap \NE(X)$. One inclusion is obvious because $\NE(\xi) = M^{\perp} \cap \NE(X) \supseteq E_0^{\perp} \cap \cdots \cap E_{c_X}^{\perp} \cap \NE(X)=L \cap \NE(X)$. Conversely, let $C \subset X$ be an irreducible curve such that $M \cdot C=0$. If $C$ does not intersect the support of $M$, then $[C] \in E_0^{\perp} \cap \cdots \cap E_{c_X}^{\perp} = L$; otherwise, $C$ must be contained in the support of $M$, hence $C \subset E_i$ for some $i=0, \ldots,c_X$. Then $C = h_i(\tilde{C})$ for some $\tilde{C} \subset \mathbb{P}^1 \times T_i$ which is contracted by $\xi_{|E_i} \circ h_i$.
Then $[C] \in (j_i \circ h_i)_* \N (\{pt\} \times T_i) = L$ by Proposition \ref{prop}(6).

Since ker$(\xi)_* \subseteq L \subseteq (E_i)^{\perp}$, by \cite[Theorem 3.7(4)]{km}, there exists a $\mathbb{Q}$-Cartier divisor $D_i$ on $S$ such that Supp$(D_i)=\Supp(\xi(E_i))$, which is one-dimensional. Hence $S$ has dimension $2$.
\medskip

Consider the morphism $\pi := (\xi,\psi): X \rightarrow S \times Y$; notice that it is finite because ker$(\psi_*) = \mathbb{R}R_0 + \cdots + \mathbb{R}R_{c_X}$, $\ker(\xi_*) \subseteq L$ and $L \cap \mathbb{R}[f_0] + \cdots + \mathbb{R}[f_{c_X}] = \{0\}$, as an easy computation shows. Moreover
\[
 \rho_X - \rho_S=\dim(\ker{\xi_*}) \leq \dim(L)=\rho_X-c_X-1;
\]
on the other hand, since $X$ dominates $S \times Y$, we get $\rho_X \geq \rho_S+\rho_Y$, from which
\[
 \rho_S \leq \rho_X - \rho_Y = c_X + 1.
\]
Hence $\rho_S=c_X+1=\rho_X-\rho_Y$. Moreover $\rho_S \leq 9$ because $S$ is dominated by the general fiber $F$ of $\psi$, which is a smooth Del Pezzo surface.

Consider the finite morphism $\xi_{|F}:F \to S$. Since $F$ is smooth, by \cite[Proposition 5.13 and Lemma 5.16]{km}, we see that $S$ has rational $\mathbb{Q}$-factorial singularities. Moreover $\xi_{|F}$ is \textquoteleft non-degenerate\textquoteright \ in the sense of \cite[Definition 1.14]{fz} and the singularities of $S$ are isolated. We can thus apply \cite[Corollary 1.27 and Note 1.26]{fz} and we see that $S$ has log-terminal singularities. By \cite[Proposition 4.18]{km}, this means that, locally around every singular point, $S$ is a quotient of $\mathbb{C}^2$ by the action of a finite group.

Finally, \cite[Corollary 7.4]{k86} shows that also $Y$ has rational singularities.

\medskip
\noindent\textbf{Second case: $\varphi$ is finite on $D_0$.} In this case $\dim(Y)=n-1$ and the fibers of $\varphi$ are all one-dimensional. By \cite[Corollary 1.9 and Theorem 4.1(2)]{AW}, the general fiber $F$ of $\varphi$ is isomorphic to $\mathbb{P}^1$, $-K_{X_s} \cdot F =2$ and, if $C$ is an irreducible component of a fiber, then its reduced structure is isomorphic to $\mathbb{P}^1$. Recall that, by Lemma \ref{ant_deg}, $-K_{X_s} \cdot C \geq 1$; therefore the arbitrary fiber $F$ of $\varphi$ can be of two types:

\begin{itemize}
\item  $F$ is irreducible without multiple components such that $F_{\red} \simeq \mathbb{P}^1$ and $-K_{X_s} \cdot F =2$;
\item $F = C \cup C'$ with $C$ and $C'$ (eventually coincident) irreducible curves without multiple components such that $C_{\red} \simeq C'_{\red} \simeq \mathbb{P}^1$ and $-K_{X_s} \cdot C = -K_{X_s} \cdot C' =1$.
\end{itemize}

Let us call a \textit{generalized conic bundle} every morphism whose fibers are all as above. Notice that the main difference from smooth conic bundles is that the fibers here are allowed to have embedded points; equivalently, $\varphi$ does not need to be flat.
\medskip

Write $\sigma$ as a composition of $s$ divisorial contractions of type $(n-1,n-2)^{eq}$
\[
\xymatrix{
 X=X_0 \ar@/^1pc/[rrrr]^{\sigma} \ar[r] & X_1 \ar[r] & \cdots \ar[r] & X_{s-1} \ar[r] & X_s;
}
\]
then for every $i=0,\ldots,s-1$, the composition $X_i \rightarrow X_{i+1} \rightarrow \cdots \rightarrow X_{s} \stackrel{\varphi}{\rightarrow} Y$ is also a generalized conic bundle, in particular for every curve $C$ contained in a fiber the intersection $-K_{X_i} \cdot C$ is integral. According to Lemma \ref{ant_deg}, this imply, as in the smooth case (see \cite[proof of Lemma 2.8]{cas10}), that $H_1:=\psi(E_1),\ldots,H_s:=\psi(E_s) \subset Y$ are pairwise disjoint. The situation now is exactely the same as in \cite[3.2.24]{cas10}. In particular we can see that $s=c_X-1$ and that there exist extremal rays $\hat{R}_1=\mathbb{R}_{\geq0}[\hat{f}_1],\ldots,\hat{R}_{c_X-1}=\mathbb{R}_{\geq0}[\hat{f}_{c_X-1}]$ in $\NE(X)$ such that:
\begin{itemize}
 \item $\mathbb{R}[\hat{f}_1],\ldots,\mathbb{R}[\hat{f}_{c_X-1}]$ are linearly indipendent in $\N(X)$;
 \item $NE(\psi)=R_1+\cdots+R_{c_X-1}+\hat{R}_1+\cdots+\hat{R}_{c_X-1}$;
 \item for $i=1,\ldots,c_X-1$, if we set $\hat{E}_i=\Locus(\hat{R_i})$, then $E_1 \cup \hat{E}_1, \ldots, E_{c_X-1} \cup \hat{E}_{c_X-1}$ are pairwise disjoint and $\psi^{*}(H_i)=E_i + \hat{E}_i$;
 \item $\mathbb{R}[f_1],\ldots,\mathbb{R}[f_{c_X-1}],\mathbb{R}[\hat{f}_1]$ are linearly indipendent in $\N(X)$ and the space they generate contains $\mathbb{R}[\hat{f}_{2}],\ldots,\mathbb{R}[\hat{f}_{c_X-1}]$;
 \item for every $i=1,\ldots,c_X-1$, there exist finite morphisms $\hat{T}_i \to \hat{\varphi}_i(\hat{E}_i)$ and $\hat{h}_i:\mathbb{P}^1 \times \hat{T}_i$ making the following diagram commute:
\[
\xymatrix{
\mathbb{P}^1 \times \hat{T}_i \ar[r]^{\hspace{.3cm}\hat{h}_i} \ar[d] & \hat{E}_i \ar[d]^{\hat{\varphi}_{i|\hat{E}_i}}\\
\hat{T}_i \ar[r] & \hat{\varphi}_i(\hat{E}_i)
}
\]
where $\hat{T}_i$ is an $(n-2)$-dimensional variety and $\mathbb{P}^1 \times \hat{T}_i \rightarrow \hat{T}_i$ is the trivial $\mathbb{P}^1$-bundle. Furthermore $(\hat{j}_i \circ h_i)_*\N(\{pt\} \times T_i) =L$, where $\hat{j}_i$ is the inclusion of $\hat{E}_i$ in $X$. 
\end{itemize}

We need the following result, whose proof is the same of \cite[Lemma 3.2.25]{cas10}.

\begin{lem} \label{ciao}
 Let $E$ be a projective variety and $\pi:E \rightarrow W$ a $\mathbb{P}^1$-bundle with fiber $f \subset E$. Let $\psi_0:E \rightarrow Y$ be a morphism onto a normal, projective and $\mathbb{Q}$-factorial variety such that $\dim (\psi_0(f))=1$. Let $H \subset Y$ be a prime divisor such that $\N(H,Y) \subsetneq \N(Y)$ and $\psi_0^{*}(H) \cdot f > 0$.

Then there exists an elementary contraction $\zeta: Y \rightarrow Y'$, with one-dimensional fibers, which makes the following diagram commute: 
\[
\xymatrix{
E \ar[r]^{\psi_0} \ar[d]_{\pi} & Y \ar[d]^{\zeta}\\
W \ar[r] & Y'
}
\]
\end{lem}

We are going to apply Lemma \ref{ciao} to the trivial $\mathbb{P}^1$-bundle $\mathbb{P}^1 \times T_0 \xrightarrow{q_0} T_0$, with $\psi_0:=\psi_{|E_0} \circ h_0: \mathbb{P}^1 \times T_0 \rightarrow Y$ and $H:=H_1=\psi(E_1)=\psi(E_0 \cap E_1)$. In fact, if $g_0$ is a fiber of $q_0$, then $(h_{0})_*([g_0])=r[f_0]$ for some positive integer $r$, and thus:
\[
 \psi_0^*(H_1)\cdot g_0 = h_0^* \circ (\psi_{|E_0})^*(H_1) \cdot g_0 = r \psi^*(H_1) \cdot f_0 = r (E_1 + \hat{E}_1) \cdot [f_0] > 0.
\]
Moreover $\N(H_1,Y) \subseteq H_2^{\perp} \subsetneq \N(Y)$; Lemma \ref{ciao} can thus be applied and we find an $(n-2)$-dimensional variety $Y'$ and a morphism $\zeta:Y \rightarrow Y'$ which is the contraction of $\mathbb{R}_{\geq 0}[\psi_0(g_0)]=\mathbb{R}_{\geq 0}[\psi(f_0)]$. Define $\psi':=\zeta \circ \varphi \circ \sigma: X \rightarrow Y'$. Then $\NE(\psi')=R_0 + R_1 + \cdots + R_{c_X-1} + \hat{R}_1 + \cdots + \hat{R}_{c_X-1}$ and $\rho_{Y'}=\rho_X - c_X - 1$. 

Consider the divisor
\[
 M'=2 E_0 + 2\sum_{i=1}^{c_X-1}E_i + \sum_{i=1}^{c_X-1}\hat{E}_i.
\]
Similarly as for the divisor $M$ of the first case, it is easy to verify that $M'$ is nef; hence it defines a contraction $\xi':X \rightarrow S'$. Exactly as in the first case, thanks to the existence of the finite morphisms $h_i$ and $\hat{h}_i$ and their properties, we see that
\[
\NE(\xi')=L \cap \NE(X).
\]
and that $\dim(S')=2.$

Define the map $\pi:=(\xi',\psi'):X \rightarrow S' \times Y'$. After checking that $L \cap \mathbb{R}[f_0]+\mathbb{R}[f_1]+ \cdots +\mathbb{R}[f_{c_X-1}]+\mathbb{R}[\hat{f}_1]+\cdots + \mathbb{R}[\hat{f}_{c_X-1}] =\{0\}$, we see that $\pi$ is finite. The thesis now follows exactly as in the first case.
\end{proof}

\begin{rmk}\label{3-4}
Let $\xi:X \to S$ be the first component of the morphism $\pi$.

If $n=3$, the general fiber of $\xi$ is a smooth Fano variety of dimension one, \textit{i.e.} it is isomorphic to $\mathbb{P}^1$. Since it dominates $Y$, which is normal, we conclude that $Y \simeq \mathbb{P}^1$.

If, instead, $n=4$, the general fiber of $\xi$ is a smooth Del Pezzo surface. Since it dominates $Y$, we see that $\rho_Y \leq 9$. Moreover we can repeat the reasoning we did for $S$ and conclude that $Y$ has log-terminal singularities.
\end{rmk}

\section{Complements}

The following remark, together with Theorem \ref{teo_princ}, implies Theorem \ref{dim3}.

\begin{rmk}
 Let $X$ be a three-dimensional $\mathbb{Q}$-factorial Gorenstein Fano variety whose singularities are canonical and isolated. Then there exists a prime divisor $D \subset X$ such that $\dim \N(D,X) \leq 2$.
\end{rmk}

\begin{proof}
Let $R_1, \ldots, R_m$ be the extremal rays of $\NE(X)$. Assume $m > 2$ (and hence $\rho_X >2$), the statement being clear otherwise. For every $i=1,\ldots,m$, call $\varphi_i:X \to Y_i$ the contraction of the ray $R_i$.

Suppose that, for some $i$, the contraction $\varphi_i$ is birational; then, by Theorem \ref{teo}, it is divisorial. Let $E$ be its exceptional divisor and $f \subset X$ a one-cycle such that $\mathbb{R}_{\geq0}[f]=R_i$. If $\dim(\varphi_i(E_i))=0$, then $\N(E,X)= \mathbb{R}[f]$ and $\dim \N(E,X)=1$. Otherwise, $\varphi_i(E)=C$ is an irreducible curve, $(\varphi_i)_*\N(E,X)\simeq\N(C,Y_i)$ and $\dim \N(E,X)=2$.

We can thus suppose that each $\varphi_i$ is of fiber type. Since we are assuming $\rho_X>2$, $\dim(Y_i)=2$ for every $i=1,\ldots,m$; moreover, the contraction of any two-dimensional face of $\NE(X)$ leads to a one-dimensional variety. Then a general fiber of such contractions is a prime divisor $F$ such that $\dim\N(F,X)=2$.
\end{proof}

\begin{rmk}
 In the setting of Theorem \ref{dim3}, when $\rho_X \geq 6$, we have $c_X\geq 4$. Then, by Theorem \ref{teo_princ}, there exists a finite morphism $\pi:X \to S \times \mathbb{P}^1$, where $S$ is a normal surface with rational quotient singularities and $\rho_S=\rho_X-1 \leq 9$.
\end{rmk}

In dimension $4$, the following result follows directly from Theorem \ref{teo_princ} and Remark \ref{3-4}.

\begin{cor}\label{dim4}
Let $X$ be a four-dimensional $\mathbb{Q}$-factorial Gorenstein Fano variety with canonical singularities and such that the closed set of non-terminal singularities is finite. Suppose moreover that there exists a prime divisor $D \subset X$ such that $\codim \N(D,X)\geq 4$; then $\rho_X \leq 18$. 
\end{cor}

In dimensions $3$ and $4$, if $X$ is not smooth, then $\pi$ cannot be an isomorphism, as the following remark shows.

\begin{rmk}
Let $X$ be a Fano variety of dimension $3$ or $4$ satisfying \eqref{X} and such that there exists a prime divisor $D \subset X$ with $\codim \N(D,X)\geq 4$. Let $\pi$ be the finite morphism of Theorem \ref{teo_princ}. Then, if $\pi$ is an isomorphism, $X$ is smooth. Indeed, suppose $X \simeq S \times Y$. If $Y$ has a singular point $y$, then $S \times \{y\} \subseteq X_{\sing}$ and $\dim(S\times \{y\})=2$, which is impossible. Similarly, if $s$ is a singular point of $S$, then $\{s\} \times Y \subseteq X_{\sing}$ and $\dim(\{s\} \times Y)=\dim(Y)=\dim(X) -2$.  
\end{rmk}

The following remark is a generalization to the singular case of \cite[Proposition 5]{tsu}.

\begin{rmk}
Let $X$ be a Fano variety satisfying \eqref{X}. If $n \geq 3$ and $c_X=\rho_X-1$, then $\rho_X\leq 3$. Indeed, let $D \subset X$ be a prime divisor with $\dim\N(D,X)=1$ and let $E_1,\ldots,E_s$ the divisors obtained running a Mori program for $-D$, as in Lemma \ref{teomp2}. Recall that $E_i \cap E_j = \emptyset$ and $E_i \cap D \neq \emptyset$ for every $i,j \in \{1,\ldots,s\}$ with $i \neq j$. Suppose $s \geq 2$. Since $n \geq 3$, we can find a curve $C_1 \subseteq E_1 \cap D$; let now $C_2 \subset D$ be a curve intersecting $E_2$ such that $C_2 \nsubseteq E_2$. Then $E_2 \cdot C_1 = 0$ and $E_2 \cdot C_2 >0$. But this is impossible because, by assumption, $C_1$ and $C_2$ are numerically proportional. Hence $s \leq 1$ and $c_X=\codim \N(D,X) \leq 2$.
\end{rmk}

We conclude this paper with an example which shows that some of the assumptions of Theorem \ref{teo_princ} cannot be omitted. We show that there exists a singular Del Pezzo surface not satisfying some of the assumptions of the theorem, for which the main statement does not hold. More precisely, this surface has log-terminal non-Gorenstein singularities and its Picard number is $10$. In general, for Del Pezzo surfaces with log-terminal singularities, the Picard number is bounded by a constant which depends only on the index (\cite[Theorem 0.1]{nik2}). When the index is one, we can take this constant to be $9$ (\cite[Proposition 3.2]{nik1}).

\begin{ex}\label{esempio} This example was found using the classification of toric log Del Pezzo surfaces of index at most $16$ in \cite{kkn}; the list of such surfaces is available in the Graded Ring Database \cite{bro}.
Let us consider in $\mathbb{R}^2$ the fan $\Sigma$ whose rays are generated by the following vectors:
\[
(-2,3), (-1,3), (1,2), (2,1), (3,-1), (3,-2), (2,-3), (1,-3), (-1,-2), (-2,-1), (-3,1), (-3,2).
\]
The toric surface $S$ defined by $\Sigma$ is a ($\mathbb{Q}$-factorial) Del Pezzo surface with log-terminal singularities (see, for example, \cite[Remark 6.7]{dais}); moreover one can check that its index is $15$. We have
\[
\rho_S=(\textmd{number of rays of $\Sigma$}) - 2=10.
 \]
Since every prime divisor of $D \subset S$ is a curve, we get $\codim\N(D,S)=\rho_S-1=9$.
\end{ex}

\providecommand{\bysame}{\leavevmode\hbox to3em{\hrulefill}\thinspace}
\providecommand{\MR}{\relax\ifhmode\unskip\space\fi MR }
\providecommand{\MRhref}[2]{%
  \href{http://www.ams.org/mathscinet-getitem?mr=#1}{#2}
}
\providecommand{\href}[2]{#2}

\end{document}